\documentclass[11pt, reqno]{amsart}

\evensidemargin
\oddsidemargin
\usepackage{amsmath, amsthm, latexsym, amssymb, bbm}
\usepackage[dvips]{color}

\usepackage{calc,  subfigure, bm}
\usepackage[dvips]{graphicx}

\newcommand{\eps}{\varepsilon}

\newcommand{\e}{\varepsilon}

\renewcommand{\phi}{\varphi}

\newcommand{\J}{{\mathcal J}}

\newcommand{\C}{\mathbb C}
\newcommand{\R}{\mathbb R}

\newcommand{\N}{\mathbb N}

\newcommand{\E}{\mathbb E}
\renewcommand{\P}{\mathbb P}
\newcommand{\M}{\mathcal{M}}

\newtheorem{theorem}{Theorem}
\newtheorem{lemma}[theorem]{Lemma}
\newtheorem{corollary}[theorem]{Corollary}
\newtheorem{prop}[theorem]{Proposition}
\newcommand{\heap}[2]  {\genfrac{}{}{0pt}{}{#1}{#2}}
\newcommand{\sfrac}[2] {\mbox{$\frac{#1}{#2}$}}

\newcounter{remnr}
{\nopagebreak {\hfill{$\diamond$}}\\ }

\renewcommand{\phi}{\varphi}

\renewcommand{\P}{\mathbb{P}}
\renewcommand{\E}{\mathbb{E}}

\begin{document}


\title[Galton-Watson trees with vanishing martingale limit]
{\Large Galton-Watson trees\\ with vanishing martingale limit}

\author[Nathana\"el Berestycki, Nina Gantert, Peter M\"orters, Nadia Sidorova]
{Nathana\"el Berestycki, Nina Gantert, Peter M\"orters, Nadia Sidorova}

\maketitle

\vspace{0.5cm}

\begin{quote}{\small {\bf Abstract: }
We show that an infinite Galton-Watson tree, conditioned on its martingale limit being smaller than $\eps$, agrees
up to generation $K$ with a regular $\mu$-ary tree, where $\mu$ is the essential minimum of the offspring
distribution and the random variable $K$ is strongly concentrated near an explicit deterministic function
growing like a multiple of $\log(1/\eps)$.  More precisely, 
we show that if $\mu\ge2$ 
then with high probability as $\eps \downarrow 0$,  $K$
takes exactly one or two values. This shows in particular that
the conditioned trees converge 
to the regular $\mu$-ary tree, providing an example of entropic repulsion where the limit has vanishing entropy.}
\end{quote}

\vspace{0.5cm}


\vspace{2cm}

\section{Introduction}

The problem of \emph{conditioning principles} can be formulated in the following way: Given that a system comprising a large number
of individual components  shows highly unlikely \emph{collective} behaviour, describe the conditional law of an \emph{individual} component.
This situation arises frequently in statistical mechanics, when an ensemble of particles is subject to some constraint (for example a fixed energy
per particle). The distribution of the individual feature given the constraint is then referred to as the \emph{micro-canonical} distribution of the system.
The most famous result in this respect is the \emph{Gibbs conditioning principle}, which loosely speaking says that
under the condition that the empirical measure
$$L_n=\frac1n\sum_{i=1}^n \delta_{X_i}$$
of a family of independent random variables $X_1, ,\ldots, X_n$ with law $P$ belongs to some convex set~$A$,
the law of $X_1$ converges to the probability measure~$Q$ that minimizes the relative entropy $H( Q \,| P)$
subject to the constraint $Q\in A$. 
There exist several refinements of this result
describing rigorously the precise asymptotic strategy by which the random variables realize the large
deviation event $\{L_n\in A\}$. See the book of Dembo and Zeitouni~\cite{dzbook} for more on the classical Gibbs conditioning principle,
\cite{C, DZ, SZ} for refinements, and \cite{DSZ, MV, MN} for further examples of conditioning principles.
\medskip
\pagebreak[3]

The present paper describes such a conditioning principle in the case of Galton--Watson trees with a nondegenerate
offspring variable~$N$ satisfying $P(N=0)=0$ and $EN \log N<\infty$. Let $a:=EN> 1$ be the mean offspring number.
We denote by $(Z_n \colon n=0,1,\ldots)$ the sequence of generation sizes of the Galton Watson tree and note
that by definition $Z_0=1$. By the Kesten-Stigum theorem the \emph{martingale limit} 
$$W:=\lim_{n\to\infty} \frac{Z_n}{a^n}$$
is well-defined and strictly positive almost surely.  Note that $W$
can be seen as a random constant factor in front of a deterministic exponential  growth term~$a^n$, which together determine
the leading order asymptotics of the generation size~$Z_n$. In the framework of the preceding paragraph
the quantity $W$ represents the collective behaviour of the branching
individuals and we are interested in the offspring distribution of individual particles given that $W$ is smaller than
a small parameter $\eps$.
\medskip

An important observation is that the offspring distribution of the conditioned tree is not uniform over all generations
and the influence of the initial generations far outweighs that of later generations. Indeed, we show
that there is a sharp threshold level  $\gamma(\eps)$, satisfying
$$\gamma(\eps) \sim \frac{\log(1/\eps)}{\log(a/\mu)},$$
such that all individuals up to generation $\lceil \gamma(\eps)\rceil-2$  only produce the minimal
number $\mu := \min\{ n\in\N \colon P(N=n)>0 \}$ of offspring. Here $\sim$ denotes that the ratio of the
left and right hand side converges to one. Decomposing the population according to their ancestry in generation~$k$ gives
$$W= \frac1{a^k} \sum_{j=1}^{Z_k} W_j,$$
where $W_1, W_2, \ldots$ are independent copies of~$W$. Using this decomposition for $k=\lceil \gamma(\eps)\rceil-1$ and
assuming that the tree performs unconditionally from generation $k$ onwards shows that $W\sim (\mu/a)^k$ and hence
$\log W \sim \log \eps$, showing that minimal branching up to generation $\lceil \gamma(\eps)\rceil-2$ almost single-handedly
delivers the collective requirement.
\medskip

Our main results confirm and substantially refine this rough picture in the case where the minimal offspring number
satisfies $\mu>1$. In this case we can describe $\gamma(\eps)$ precisely as
\begin{align*}
\gamma(\e):= \frac{\log(1/\e)}{\log(a/\mu)} -\frac{\log\log(1/\e)}{\log\mu}+H(\e),
\end{align*}
where $H$ is a multiplicatively periodic continuous function with period $a/\mu$.
The first branching producing more than the minimal number of offspring occurs in
generation $\lceil \gamma(\eps)\rceil-1$ or $\lceil \gamma(\eps)\rceil$.
We show that for most values of $\eps$ it
occurs in generation $\lceil \gamma(\eps)\rceil-1$ and, defining the random variable
$$K:=\min\big\{k\in\N \colon Z_k>\mu^k\big\},$$
we find that the size of generation~$K$ is asymptotically still given by $\mu^{K}$ with a relatively small additive
$\eps$-dependent correction. Before describing these results in more detail in Section~2, we now briefly explain
the situation in the `degenerate' case $\mu=1$, in which nonexponential growth of the tree is possible. The
concentration effect of the random variable $K$ which holds in the case $\mu>1$ is much less pronounced in the
case~$\mu=1$, but the result can be obtained by soft arguments, whereas the case of general $\mu>1$ requires much more subtle reasoning.
\medskip

\pagebreak[4]

\section{Statement of the main results}

We start by describing our results in the case $\mu=1$, for which the analysis is fairly straightforward. 
In this case we define
$$\gamma(\eps) := \frac{\log(1/\eps)}{\log a}.$$ Our results in this case are summarised by the following proposition.

\begin{prop}\label{Prop:mu=1}
There exists $\lambda>0$ such that
\begin{equation}\label{mu=1result}
\limsup_{\eps\downarrow 0} \P\big(|K-\gamma(\eps)| \geq x \big| W<\eps \big) \leq e^{-\lambda x}\quad
\mbox{ for all $x\geq 1$.}
\end{equation}
\end{prop}

In other words the time of the first branching producing more than the minimal number of offspring occurs in
generation $\gamma(\eps)$ with a tight random correction of order one. 

\medskip

Because $K\uparrow\infty$ this implies that the Galton-Watson tree conditioned
on $W<\eps$ converges (in a sense detailed below) for $\eps\downarrow 0$ to the regular $\mu$-ary tree. This fact, which also
holds in the case~$\mu>1$, is quite remarkable when seen in a large deviations context.
We shall explain this further in the next section, after the first main result is established.

\medskip 
We now come to the main result of this paper, which deals with the case $\mu\ge 2$. More precisely, we consider a Galton--Watson tree with offspring probabilities $p_n=P(N=n)$ and keep the notation established
above. We assume that $\mu=\min \{ n\in\N \colon p_n>0 \} \ge 2$ and also exclude the trivial case $p_{\mu}=1$.
Recall that $K-1$ is the first generation where an individual has more than the minimal number of offspring.

\begin{theorem}
\label{t_main}
We have
\begin{align*}
\lim_{\eps\downarrow0} \P\big(\, K = \lceil\gamma(\e)\rceil \ \text{\rm  or } K= \lceil\gamma(\e)\rceil +1 \, \big| \,  \ W<\e\big) = 1,
\end{align*}
where
\begin{align*}
\gamma(\e):=
\frac{\log(1/\e)}{\log(a/\mu)}
-\frac{\log\log(1/\e)}{\log\mu}+H(\e)
\end{align*}
and $H$ is a multiplicatively periodic continuous nonrandom function with period $a/\mu$.
\end{theorem}
\medskip

Before giving more detailed results on the shape of the conditioned tree, we give an interpretation of~Theorem~\ref{t_main} and
put it into context. 
To this end we denote by $\mathcal T$ the space of all rooted trees with the
property that every vertex has finite degree. A metric~$d$ on this space is uniquely determined by the requirement that $d(T_1,T_2)=e^{-n}$,
when $n$ is maximal with the property that the trees $T_1$ and $T_2$ coincide up to the $n$th generation. This makes $(\mathcal T,d)$ a
complete, separable metric~space. The next results also holds when~$\mu=1$.%
\medskip%

\begin{corollary} \label{T}
As $\eps\downarrow 0$, conditionally on the event $\{W<\eps\}$ the tree $T$ converges in law on
$(\mathcal T, d)$ to the regular $\mu$-ary tree, i.e., the tree in which every vertex has exactly $\mu$ offspring.
\end{corollary}

\begin{proof}
The statement is equivalent to
$\lim_{\eps\downarrow 0} \P( Z_k=\mu^k \, | \, W<\eps) = 1$, for all $k\in\N$.
This follows directly from~\eqref{mu=1result} in the
case $\mu=1$, and from Theorem~\ref{t_main} in the case $\mu>1$.
\end{proof}
\medskip

From the point of view of large deviations theory this result is quite surprising, at least at a first glance.
One would expect that the limiting behaviour represents the optimal strategy by which the event $W=0$ is realized  and
that this strategy depends on the details of the law of~$N$. 
There seems to be no good
reason why in the limit the \emph{growth rate} of the tree should drop dramatically, or in fact why it should drop at all,
as we only require the constant to be small.  Above all, the probability of seeing a $\mu$-ary tree up to the $n$th generation
may be arbitrarily small and can certainly be much smaller than those of seeing other trees satisfying~$Z_n\leq \eps a^n$.
\medskip%

This becomes even more intriguing if the result is put in the context of \emph{entropic repulsion}, an expression used by physicists to convey
the idea that entropy maximisation may force certain systems to obey properties that are not obviously imposed on them \emph{a priori}. This phenomenon has been
studied mathematically by Bolthausen et al~\cite{bdg} in the context of the two-dimensional harmonic crystal with hard wall repulsion,
and by Benjamini and Berestycki~\cite{BB} and~\cite{bb2}, where it is shown that conditioning a one-dimensional Brownian motion on some self-repelling behaviour may force
the process to satisfy a strongly amplified version of the constraint. Usually, the reason entropic repulsion may arise is in order to increase the entropy of the
system, i.e., make room for fluctuations. Thus the eventual state of the system is a compromise between the energy cost of adopting an unusual behaviour and the
entropic benefits. Corollary~\ref{T} may be cast in this framework, as it shows that the effect of requiring the constant $W$ to be small is to reduce the
overall exponential growth rate from $a$ to $\mu$. If the limiting state of the system is non-random, as it is the case in our model, what could the
entropic benefits possibly be?
\medskip

The resolution of this apparent paradox comes from understanding the inhomogeneity
of the optimal strategy. While the growth rate $\log a$ is purely asymptotic, i.e.\ depends only on the offspring numbers
\emph{after} any given generation, the growth constant $W$ depends heavily on the \emph{initial} generations of the tree.
Roughly speaking, the collection of trees which form the optimal strategy to achieve $W<\eps$ have minimal offspring for roughly $\gamma(\eps)$
generations, which causes high entropic and energetic cost but only for a small number of generations,
and then after a while switch to  growth with the natural rate~$\log a$. The initial behaviour ensures that $W$ is small
at a minimal probabilistic cost, because for all but a small number of generations the trees
can have their natural growth. The topology on $\mathcal T$ compares trees starting from their root
so that in the limit we only see the behaviour in the initial generations. This leads to a limiting
object with minimal growth rate at all generations and creates the illusion of a drop in the growth
rate for the optimal strategy. A somewhat similar phenomenon is observed by Bansaye and Berestycki \cite{BanBer}
in the context of branching processes in random environment, although they consider situations where the growth rate
is directly conditioned to be atypical.
\bigskip

In the following two theorems we return to the case $\mu>1$ and take a closer look at the \emph{shape of the conditioned tree}
and thus on the inhomogeneous strategy underlying
the conditioning event $W<\eps$.
Figure~1 sketches the curve $\gamma$ and, for each $\e$, the two possible values for $K$, namely
$\lceil\gamma(\e)\rceil$ and $\lceil\gamma(\e)\rceil+1$, represented by the horizontal lines.
Roughly speaking, we will see in Theorem~\ref{t_main2} that for most $\e$ the random variable~$K$ has a particular non-random
value, represented by the thick horizontal lines.  For most values of $\eps$ we have  $K=\lceil\gamma(\e)\rceil$ and only very occasionally
$K=\lceil\gamma(\e)\rceil+1$. The switch happens when  $\gamma(\e)$ gets too close to the integer $\lceil\gamma(\e)\rceil$. Then, for a short
range of values of~$\e$, marked in grey on the zoomed picture, $K$ is truly random and can take the values $\lceil\gamma(\e)\rceil$
and $\lceil\gamma(\e)\rceil +1$. As $\eps$ decreases further $\lceil\gamma(\e)\rceil$ loses its power, and $K$ moves to $\lceil\gamma(\e)\rceil +1$.
This, in turn, does not last long  because when $\eps$ decreases just a little more the curve $\gamma$ crosses an integer level, and then
for another long range of $\e$ the random variable~$K$ takes the value~$\lceil\gamma(\e)\rceil$ again.%
\medskip%

\begin{figure}[h]
\label{pic}
\scalebox{0.8}{\input{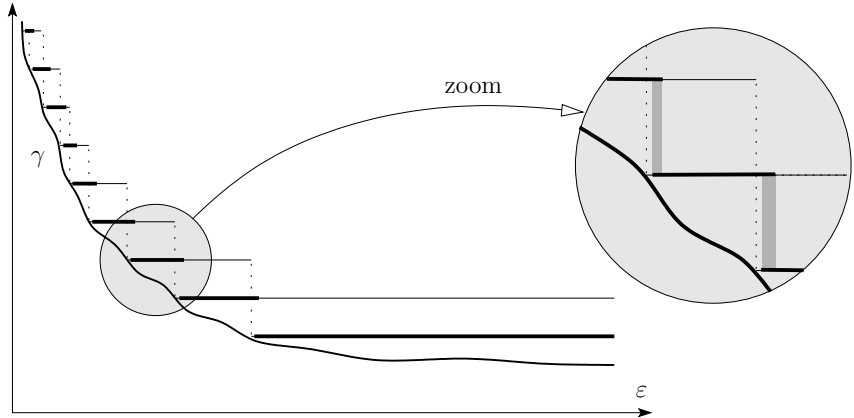}}
\caption{The time of the first branching with more than the minimal offspring.}
\end{figure}

In order to be able to formulate this precisely, we need to identify the different regions of $\e$.
Let $\beta:=\log\mu/\log a$ and $\alpha:=\frac{\beta}{1-\beta}$, and denote
$\{x\}:=\lceil x\rceil -x$, for $x\in\R$.
Further, denote
\begin{align*}
\omega(\e):=\frac{\e^{\alpha(\mu^{-\{\gamma(\e)\}}-1)}}{\log(1/\e)}.
\end{align*}
It is easy to see that
$$\liminf\limits_{\e\downarrow 0} \omega(\e)=0
\qquad\text{and}\qquad \limsup\limits_{\e\to 0} \omega(\e)=\infty,$$
where the liminf is achieved via values of $\e$ for which $\{\gamma(\e)\}$ is very small,
and the limsup is achieved via most other values of $\e$.
%

In the sequel, if $(q_j)$ is a sequence of
probabilities we write $0 \prec q_j \prec 1$ to denote that the sequence is asymptotically bounded away from zero and one.
For sequences $(a_j)$, $(b_j)$ of positive numbers we use the symbol $a_j \asymp b_j$ to denote that $a_j/b_j$ is asymptotically
bounded away from zero and infinity.
\medskip


\begin{theorem}\ \\[-5mm]
\label{t_main2}

\begin{itemize}
\item[(a)] Suppose $\eps_j\downarrow 0$ such that $\omega(\eps_j)\to\infty$. Then
\begin{align*}
\lim_{j\to \infty}\P\big(\, K = \lceil\gamma(\e_j)\rceil \big | \ W<\e_j\big)=1.
\end{align*}
\item[(b)]  Suppose $\eps_j\downarrow 0$ such that $\omega(\eps_j)\asymp 1$,
then
\begin{align*}
0 \prec  \P\big(\, K = \lceil\gamma(\e_j)\rceil \, \big| \, W<\e_j\big),
\P\big(\, K = \lceil\gamma(\e_j)\rceil+1 \, \big| \, W<\e_j\big) \prec 1.
\end{align*}
\item[(c)] Suppose $\eps_j\downarrow 0$ such that $\omega(\eps_j)\to0$. Then
\begin{align*}
\lim_{j\to\infty} \P\big(\, K = \lceil\gamma(\e_j)\rceil+1 \, \big| \, W<\e_j\big)=1.
\end{align*}
\end{itemize}
\end{theorem}
\smallskip

{\bf Remark:}
It is possible to compute the exact asymptotics in the second regime but we do not want to overload
the paper with unpleasant computations.
\medskip

Next, we address the question of what happens in the generation where the first non-minimal branching occurs.
We denote
\begin{equation}\label{lambdadef}
\lambda:=\min \{ n>\mu \colon p_n>0 \}.
\end{equation}
\medskip

\begin{theorem}
\label{t_prop}
If $\e_j \downarrow  0$ such that  $\omega(\eps_j)\to\infty$ or $\omega(\eps_j)\to0$, then
\begin{align*}
\lim_{j\to\infty}\frac{Z_{K}-\mu^{K}}{\mu^{K}\e_{j}^{\alpha \mu^{\gamma(\eps_j)-{K}}}}
= \big(\sfrac{\lambda}{\mu}-1\big)p_{\lambda}p_{\mu}^{-\frac{\lambda-1}{\mu-1}},
\end{align*}
in probability under $\P( \,\cdot\, | \, W<\eps_j)$.
\end{theorem}
\smallskip

{\bf Remarks:}
\begin{itemize}
\item[(a)] The influence of the first extra branching on the next generation is very small. Roughly, in generation $K-1$  most of the  individuals
still have the minimal  number~$\mu$ of children and only a small proportion of order $\e^{\alpha \mu^{\gamma(\eps)-K}}$ have more than $\mu$ children.
It can be seen from the proof that most of these individuals would have 
exactly $\lambda$~children.  \\[-2mm]
\item[(b)]  Not only does $\omega(\eps)$ govern
the transition between the regimes, it also explicitly controls the number of additional children.
Indeed, in regime (a) in Theorem \ref{t_prop}, when $K=\lceil \gamma(\eps) \rceil$, the number
of extra individuals in generation~$K$~is of order
\begin{align}
\label{regime}
\mu^{K}\e^{\alpha \mu^{\gamma(\eps)-K}}
= \mu^{\gamma(\eps)+\{\gamma(\eps)\}}\e^{\alpha \mu^{-\{\gamma(\eps)\}}},
\end{align}
which is bounded from above and below by constant multiples of $\omega(\eps)$.
This number can be quite large, 
but as we approach the end of the regime
the number 
of extra individuals becomes smaller.  Eventually, there are
no extra individuals which means that there is no more extra branching at time $\lceil \gamma(\eps) \rceil-1$,
and the point of transition moves to $K=\lceil \gamma(\eps) \rceil+1$.  \\[-2mm]
\item[(c)] We conjecture that  the extra branching remains negligible for a few generations (corresponding roughly to
the second term in the definition of $\gamma(\eps)$) and after that the tree starts growing at its normal rate.
\end{itemize}


\section{Proof of Proposition \ref{Prop:mu=1}}
\label{S:proofmu=1}

To prove~\eqref{mu=1result},
decompose the population according to their ancestry in generation~$K$ and~get
\begin{equation}\label{keyrecursion}
W= \frac1{a^K} \sum_{i=1}^{Z_K} W_i =: \frac1{a^K} W',
\end{equation}
where $W_i$ are independent copies of $W$, independent of $Z_K$ and $K$. Note that, as $\mu=1$, the random variable $K$
is independent of $Z_K$  and hence of $W'$. Using the abbreviation $$p_n:=P(N=n),
\qquad \mbox{ for }n \in \N, $$
and letting $\tau:= -\log p_1/\log a$ we get from \cite{Dubuc71b} or an easy argument in~\cite{ortgiese} that there exist constants
$0<c_1<C_1$ such that, for all $0<\eps<1$,
$$c_1\, \eps^\tau \leq \P(W<\eps) \leq C_1\, \eps^\tau.$$
Hence, for $\ell=\gamma(\eps)-z$, $z>0$,
\begin{equation}\begin{aligned}\label{displmu=1}
\P\big(K=\ell \,\big|\, W<\eps\big)  
& \leq c_1^{-1}\, \eps^{-\tau} \,\P( W'<\eps a^\ell)\,\P( K=\ell)\\
& \leq  c_1^{-1}p_1^{-1}\, p_1^\ell \,\eps^{-\tau} \, \P(W<\eps a^\ell)^2\\
& \leq   c_1^{-1}p_1^{-1}C_1^2 \, \exp(\ell \log p_1 +\tau \log \eps+ 2\tau \ell \log a) \\
& = c_1^{-1}p_1^{-1}C_1^2 \, \exp( (\log p_1) z),
\end{aligned}
\end{equation}
where we used (\ref{keyrecursion}) in the second inequality. Summing over all $z\geq x$ gives, for a suitable choice of $\lambda>0$,
$$\P\big(K\leq \gamma(\eps) -x \,\big|\, W<\eps \big) \leq e^{-\lambda x}.$$
Conversely note that, making $\lambda>0$ smaller if necessary,
$$\P\big(K>\gamma(\eps)+z\, \big| \,W<\eps\big)\leq c_1^{-1}\,  \eps^{-\tau} \, p_1^{\gamma(\eps)} \, p_1^z
\leq  e^{-\lambda z},$$
completing the proof of~\eqref{mu=1result}.

\section{Notation and background}

In this section we prepare the proof of our main theorems. We start by  introducing some additional
notation and background from the paper~\cite{FW09} by Fleischmann and Wachtel, on which our proofs are based.
In the sequel, we often omit the  argument $\e$ from $\gamma(\eps), \omega(\eps)$
and similar expressions to shorten the formulas.  We always assume that $\e$ is small
enough. Let
\begin{equation}
\label{kappadef}
\kappa(\e):=\left\lfloor \frac{\log(1/\e)}{\log(a/\mu)}\right\rfloor,
\end{equation}
and denote $y(\e):=\e (a/\mu)^{\kappa(\e)}\in (\mu/a, 1].$
Let  $$\phi(z):=\E e^{-zW}, \qquad \mbox{ for } z\in\C, \mathcal{R}e (z)\ge 0,$$
be the Laplace transform of $W$ and let
$$f(s):=\sum_{j=0}^{\infty}p_js^j,\qquad \mbox{ for } s\in [0,1], $$
be the offspring generating function of the Galton-Watson tree. Denote $f_0(z):=z$ and
$f_m(z):=f(f_{m-1}(z))$, for $m\in \N$. The \emph{logarithmic B\"ottcher function} is defined by
\begin{align*}
b(s):=\lim_{m\to\infty}\mu^{-m}\log f_m(s), \qquad \mbox{ for } s\in (0,1].
\end{align*}
That the limit exists in the B\"ottcher case follows for instance from Lemma~10 in \cite{FW09}.
Note that $b\circ\phi<0$ on $(0,\infty)$ and recall from Lemma~17 in~\cite{FW09} that the function
$(b\circ \phi)'$ increases from $-\infty$ to $0$ on $(0,\infty)$. Therefore, for any $q\in [1,2]$,
there exists a unique $u_q(\e)>0$ such that
\begin{equation}
\label{def_u}
(b\circ \phi)'(u_q)=-y/q,
\end{equation}
where $y = y(\eps)$ is defined under \eqref{kappadef}.
Observe that since the ranges of $y$ and $q$ are bounded we have
$u_q\in [u_*,u^*]$ for some $0<u_*<u^*$ for all $\e$ and $q$. Define
\begin{equation}
\label{sigmadef}
\sigma_q^2(\e):=\frac{d^2}{d u^2}(b\circ \phi)(u_q)>0,
\end{equation}
where the positivity follows from Lemma~17 in~\cite{FW09}.
\medskip

Let $d\in \{-1,0,1\}$ and
\begin{equation}\label{ndef}
n(\e):=\kappa(\eps)-\lceil \gamma(\e)\rceil -d\, .
\end{equation}
Observe that $n(\eps)\to \infty$, $n(\eps)/\kappa(\eps)\to 0$, and that
$\kappa-n\in \{\lceil \gamma\rceil, \lceil\gamma\rceil+1\}$ if and only if $d\in \{0,1\}$.
Note that $n$ depends on $d$. This dependence is omitted in the notation but we always make it clear if a particular
value of~$d$ is used. If no explicit assumption is made about~$d$, then it is arbitrary  (but independent of~$\e$).
Recall (\ref{lambdadef}) and denote
\begin{equation}
\label{Hdef}
H(\e):= \frac1{\log \mu}\, \log\Big(-\frac{b(\phi(u_1))y^{\alpha}(\lambda-\mu)}{\alpha}\Big).
\end{equation}
Since $y$ is continuous and multiplicatively periodic with period $a/\mu$ so is $u_1 = u_1(\eps)$ and thus so is $H$.
\bigskip

Observe that we may extend the domain of all functions $f_n$ to complex variables $z$ with $|z|\le 1$.
Denote
$\mathcal{D}(\delta,\theta):=\{z\in\C: 0<|z|\le 1-\delta, |\text{arg }z|\le \theta\},$
for $\delta\in (0,1), \theta\in (0,\pi).$
By Lemma~10 in~\cite{FW09} for every $\delta\in (0,1)$ there is $\theta\in (0,\pi)$ such that
$f_m(z)\neq 0$ for all $m$ and $z\in \mathcal{D}(\delta,\theta)$ and $b$ can be extended to an analytic function
on $\mathcal{D}(\delta,\theta)$ by the uniformly converging series
\begin{align*}
b(z)=\log z+\sum_{j=0}^{\infty}\mu^{-j-1}\log\frac{ f_{j+1}(z)}{f_j(z)^{\mu}}.
\end{align*}
Observe that on $\mathcal{D}(\delta,\theta)$ we have
\begin{align*}
\mu^{-m}\log f_m(z)=\log z+\sum_{j=0}^{m-1}\mu^{-j-1}\log\frac{ f_{j+1}(z)}{f_j(z)^{\mu}}
\end{align*}
and denote
\begin{align*}
\psi_m(z):=b(z)-\mu^{-m}\log f_m(z)-\frac{\mu^{-m}}{\mu-1}\log p_{\mu}, \qquad \mbox{ for } z\in  \mathcal{D}(\delta,\theta).
\end{align*}
It is easy to see that
\begin{align}
\label{psi}
\psi_m(z)=\sum_{j=m}^{\infty}\mu^{-j-1}\log\frac{ f_{j+1}(z)}{p_{\mu}f_j(z)^{\mu}}.
\end{align}
This implies, in particular, that $\psi_m(s)>0$ for all $s\in (0,1)$ and all $m\in \N$. Our next aim
is to describe the asymptotic behaviour of $\psi_m$ and $\psi_m'$ as $m\to\infty$.
%
In the sequel we use the Landau symbols $o(f)$ and $O(f)$ to denote nonnegative functions, whose actual definition
can change at every occurrence, with the property that when divided by~$f$ they converge to zero, respectively stay bounded
from above. By Lemma 10 in~\cite{FW09},
\begin{align}
\label{fn}
f_m(z)=p_{\mu}^{-\frac{1}{\mu-1}}\exp\big\{\mu^mb(z)+O(e^{-\delta\mu^m})\big\},
\end{align}
that is, $\psi_m(z)=O(\mu^{-m}e^{-\delta\mu^m})$
uniformly on $\mathcal{D}(\delta,\theta)$ as $m\to\infty$. In the next lemma we compute a much more precise asymptotics
for $\psi_m$.
\medskip

\begin{lemma}
\label{l_psi}
As $m\to\infty$,
\begin{align*}
\psi_m(z)&=p_{\lambda}p_{\mu}^{-\frac{\lambda-1}{\mu-1}}\mu^{-m-1}\exp\big\{(\lambda-\mu)\mu^mb(z)\big\}\, (1+o(1))
\end{align*}
and
\begin{align*}
\psi_m'(z)=\mu^m\psi_m(z)\,O(1)
\end{align*}
uniformly on compact subsets of $\mathcal{D}(\delta,\theta)$. 
\end{lemma}

\begin{proof} Using $f_{j+1}(z)=f(f_j(z))$
we obtain
\begin{align}
\label{bot}
\frac{ f_{j+1}(z)}{p_{\mu}f_j(z)^{\mu}}
&=1+\sum_{l=1}^{\infty}\frac{p_{\mu+l}}{p_{\mu}}f_j^l(z).
\end{align}
As $j\to\infty$, we have $f_j(z)\to 0$ uniformly
on $\mathcal{D}(\delta,\theta)$, and hence also
\begin{align}
\label{bbb}
\frac{ f_{j+1}(z)}{p_{\mu}f_j(z)^{\mu}}
=1+\frac{p_{\lambda}}{p_{\mu}}f_j^{\lambda-\mu}(z)(1+o(1))
\end{align}
and
\begin{align*}
\log \frac{ f_{j+1}(z)}{p_{\mu}f_j(z)^{\mu}}&
=\frac{p_{\lambda}}{p_{\mu}}f_j^{\lambda-\mu}(z)\,\big(1+o(1)\big).
\end{align*}
Substituting this into~\eqref{psi} and taking~\eqref{fn} into account we~get
\begin{align*}
\psi_m(z)
&=\frac{p_{\lambda}}{p_{\mu}}\sum_{j=m}^{\infty}\mu^{-j-1}f_j^{\lambda-\mu}(z)\,(1+o(1))\\
&=p_{\lambda}p_{\mu}^{-\frac{\lambda-1}{\mu-1}}\sum_{j=m}^{\infty}\mu^{-j-1}\exp\big\{(\lambda-\mu)\mu^jb(z)\big\}
(1+o(1))\\
&=p_{\lambda}p_{\mu}^{-\frac{\lambda-1}{\mu-1}}\mu^{-m-1}\exp\big\{(\lambda-\mu)\mu^mb(z)\big\}\,(1+o(1)).
\end{align*}
Substituting~\eqref{bot} into~\eqref{psi} and differentiating the uniformly converging series of analytic functions,
we get
\begin{align*}
\psi'_m(z)
&=\sum_{j=m}^{\infty}\mu^{-j-1}\Big(1+\sum_{l=1}^{\infty}\frac{p_{\mu+l}}{p_{\mu}}f_j^l(z)\Big)^{-1}
\sum_{l=1}^{\infty}l\frac{p_{\mu+l}}{p_{\mu}}f_j^{l-1}(z)f_j'(z).
\end{align*}
Using the leading term of the asymptotics~\eqref{bbb} and
\begin{align*}
\sum_{l=1}^{\infty}l\frac{p_{\mu+l}}{p_{\mu}}f_j^{l-1}(z)=
f_j^{\lambda-\mu-1}(z)\,O(1),
\end{align*}
as $j\to\infty$ uniformly on $\mathcal{D}(\delta,\theta)$,
we obtain
\begin{align*}
\psi'_m(z)
&=\sum_{j=m}^{\infty}\mu^{-j-1}f_j^{\lambda-\mu-1}(z)f_j'(z)\,O(1)
=\sum_{j=m}^{\infty}(\mu^{-j}\log f_j)'(z)
f_j^{\lambda-\mu}(z)\,O(1)\\
&=\sum_{j=m}^{\infty}
\exp\big\{(\lambda-\mu)\mu^j b(z)\big\}\,O(1)
=\exp\big\{(\lambda-\mu)\mu^m b(z)\big\}\,O(1)=\mu^m\psi_m(z)\,O(1),
\end{align*}
where we have used the Weierstrass theorem to justify the convergence of the derivatives
of uniformly converging analytic functions, and also absorbed a factor $b'(z)$ into $O(1)$.
\end{proof}


\section{The lower tail of sums of independent copies of $W$}

The main result in~\cite{FW09} yields the following fine lower tail behaviour of~$W$.

\begin{lemma}
\label{eps}
As $\e\to 0$, recalling (\ref{def_u}) and (\ref{sigmadef}),
\begin{align*}
\P\big(W<\e\big)
&=p_{\mu}^{-\frac{1}{\mu-1}}\frac{1}{\sigma_1 u_1\sqrt{2\pi}}\,
\mu^{-\frac{\kappa}{2}}\, \exp\big\{\mu^{\kappa}\big(b(\phi(u_1))+yu_1\big)+o(1)\big\}.
\end{align*}
\end{lemma}

\begin{proof}
By Theorem 1 in~\cite{FW09} we have, as $\e\to 0$,
\begin{align}
\label{tailb}
\P(W<\e)=L(\e)\e^{\frac{\alpha}{2}}\exp\big\{-M(\e)\e^{-\alpha}+o(1)\big\},
\end{align}
where $M$ and $L$ are positive multiplicatively periodic functions with period $a/\mu$ given by
\begin{align*}
M(\e):&=-\e^{\alpha}\min_{v>0}\big\{b(\phi(v))+v\e\big\},\\
L(\e):&=p_{\mu}^{-\frac{1}{\mu-1}}\frac{y^{-\frac{\alpha}{2}}}{\sigma_1u_1\sqrt{2\pi}},
\end{align*}
see formula (142) and (155) in~\cite{FW09}. Using $\e(a/\mu)^\kappa=y$, $a=\mu^{\frac 1 \beta}$, the definition of $u_1$, the periodicity of $M$ and the convexity of $b \circ \varphi$ (see Lemma 17 in \cite{FW09}), we have
\begin{align*}
-M(\e)\e^{-\alpha}=-M(y)\e^{-\alpha}
=y^{\alpha}\e^{-\alpha}\big(b(\phi(u_1))+yu_1\big)=
\mu^{\kappa}\big(b(\phi(u_1))+yu_1\big)
\end{align*}
and
$y^{-\frac{\alpha}{2}}\e^{\frac{\alpha}{2}}=\mu^{-\frac{\kappa}{2}},$
which completes the proof.
\end{proof}
\smallskip

Recall that in our calculation~\eqref{displmu=1} for the case $\mu=1$ we used a crude estimate to bound the lower tail probability of $W'$, the
sum of  finitely many independent copies  $W_1, W_2, \dots$ of the limiting variable $W$. While this estimate holds in general, it is insufficient in the
case $\mu>1$. The main goal of this section is to establish a fine result describing the lower tails of the sum of  independent copies of $W$ in this case.
The proof uses three technical lemmas which, for the
reader's convenience,  are stated and proved after the presentation of the main argument.%
\medskip%

\begin{prop}
\label{p_main}
As $\e\to 0$, with $\kappa$ and $n$ defined in (\ref{kappadef}) and (\ref{ndef}),
\begin{equation}\begin{aligned}
\label{asymp}
\P\Big(\sum_{j=1}^{q\mu^{\kappa-n}}& W_j<\e a^{\kappa-n}\Big)\\
& =p_{\mu}^{-\frac{q\mu^{\kappa-n}}{\mu-1}}\frac{1}{\sigma_q u_q\sqrt{2\pi q}}\,
\mu^{-\frac{\kappa}{2}}\exp\big\{\mu^{\kappa}\big(qb(\phi(u_q))-q\psi_n(\phi(u_q))+yu_q\big)\big\}\, I_q(\e),
\end{aligned}\end{equation}
uniformly in $q\in [1,2]$ such that $q\mu^{\kappa-n}\in \N$, where 
$I_q(\e)$ has the following properties:
\begin{itemize}
\item it is uniformly bounded in $q$ and $\e$;
\item if $\e_j\to 0$ such  that $\mu^{\kappa(\eps_j)}\psi_{n(\eps_j)}(\phi(u_1(\eps_j)))=O(1)$ then
$I_1(\eps_j)\to 1$.
\end{itemize}
\end{prop}

\begin{proof}
We have
\begin{align*}
\E e^{-t(W_1+\cdots+W_{q\mu^{\kappa-n}})}=\phi^{q\mu^{\kappa-n}}(t)
\end{align*}
and so by the inversion formula for distribution functions
\begin{align*}
\P\Big(\sum_{j=1}^{q\mu^{\kappa-n}}W_j<\e a^{\kappa-n}\Big)=\frac{1}{2\pi}\int_{-\infty}^{\infty}\frac{1-e^{-i\tau \e a^{\kappa-n}}}{i\tau}\,
\phi^{q\mu^{\kappa-n}}(-i\tau)\, d\tau.
\end{align*}
Changing the integration contour we get for arbitrary $p>0$
\begin{align*}
\P\Big(\sum_{j=1}^{q\mu^{\kappa-n}}W_j<\e a^{\kappa-n}\Big)
&=\frac{1}{2\pi}\int_{-\infty}^{\infty}\frac{e^{(p-i\tau) \e a^{\kappa-n}}-1}{p-i\tau}\,\phi^{q\mu^{\kappa-n}}(p-i\tau)\,d\tau.
\end{align*}
Substituting $p=u_qa^n$ and $\tau=ta^n$ and using the Poincar\'e functional
equation $\phi(az)=f(\phi(z))$, we obtain
\begin{align}
\P\Big(\sum_{j=1}^{q\mu^{\kappa-n}}W_j<\e a^{\kappa-n}\Big)
&=\frac{1}{2\pi}\int_{-\infty}^{\infty}\frac{e^{\e a^{\kappa}(u_q-it)}-1}{u_q-it}\phi^{q\mu^{\kappa-n}}(a^n(u_q-it))\,dt\notag\\
&=\frac{1}{2\pi}\int_{-\infty}^{\infty}\frac{e^{y \mu^{\kappa}(u_q-it)}-1}{u_q-it}f^{q\mu^{\kappa-n}}_n(\phi(u_q-it))\,dt.
\label{int}
\end{align}
Recall that $u_*>0$ has been fixed in such a way that $u_q\ge u_*$ for all $\e$ and $q$.
By Lemma 15 from~\cite{FW09} there is a constant $c>0$ such that
for all $\theta\in (0,c]$,
\begin{align*}
\phi(v-it)\in \mathcal{D}\big(1-\phi(u_*),\theta/c\big), \qquad\text{ for all }v\ge u_*, |t|\le \theta.
\end{align*}
By Lemma 10 from~\cite{FW09} there is $\theta_1>0$ such that for all $0<\theta<\theta_1$ the function $b$
and so all functions $\psi_n$ are analytic on $\mathcal{D}(1-\phi(u_*),\theta/c)$. This implies, in particular, that $\frac{\partial^3}{\partial t^3}b(\phi(v-it))$
is bounded on the set $\{v\ge u_*, |t|\le \theta\}$ and that the family
$\frac{\partial^3}{\partial t^3}\psi_n(\phi(v-it))$ is uniformly bounded on the set $\{v\ge u_*, |t|\le \theta\}$,
where uniformity follows from the fact that the $\psi_n$ are analytic and converge uniformly to zero.
\medskip

Expanding in a Taylor series in $t$ and using the definition of $u_q$ and $\sigma_q^2$ we get
\begin{align}
\notag
b(\phi(u_q-it))&=b(\phi(u_q))-it(b\circ \phi)'(u_q) -\frac{t^2}{2}\frac{d^2}{d u^2}(b\circ \phi)(u_q)+O(t^3)\\
&=b(\phi(u_q))+\frac{ity}{q} -\frac{\sigma^2_qt^2}{2}+O(t^3),
\label{ts}
\end{align}
and
\begin{align}
\notag
\psi_n(\phi(u_q-it))
&=\psi_n(\phi(u_q))-it(\psi_n\circ \phi)'(u_q) -\frac{t^2}{2}\frac{d^2}{d u^2}(\psi_n\circ \phi)(u_q)
+O(t^3)\\
&=\psi_n(\phi(u_q))-ita_q -\frac{s_qt^2}{2}+O(t^3),
\label{ts2}
\end{align}
with
\begin{align*}
a_q(\e):=(\psi_n\circ\phi)'(u_q)\qquad\text{ and }\qquad s_q(\e):=\frac{d^2}{d u^2}(\psi_n\circ \phi)(u_q),
\end{align*}
as $t\to 0$, uniformly in $\e$ and $q$.
Observe that $s_q\to 0$ as $\e\downarrow 0$ uniformly in $q$  and so
$\sigma^2_q-s_q>0$ for all $\e>0$ small enough and all $q$.
We fix $\theta<\theta_1$ so that  for all $t\le \theta$ the functions $O$ in~\eqref{ts} and~\eqref{ts2} satisfy
$|O(t^3)|<{\sigma_q^2t^2}/{8}$. Let
\begin{align*}
\rho(\e):=\kappa\mu^{-\frac{\kappa}{2}}.
\end{align*}
For $\e$ small enough we split the integral in~\eqref{int} as
\begin{equation}\begin{aligned}
2\pi&\P\Big(\sum_{j=1}^{q\mu^{\kappa-n}}W_j<\e a^{\kappa-n}\Big)\\
&=\int_{-\rho}^{\rho}\frac{e^{y \mu^{\kappa}(u_q-it)}}{u_q-it} \, f^{q\mu^{\kappa-n}}_n(\phi(u_q-it))\, dt
+\int_{|t|\in [\rho,\theta]}\frac{e^{y \mu^{\kappa}(u_q-it)}}{u_q-it}\,f^{q\mu^{\kappa-n}}_n(\phi(u_q-it))\, dt\\
&\qquad +\int_{|t|\ge\theta}\frac{e^{y \mu^{\kappa}(u_q-it)}}{u_q-it}\,f^{q\mu^{\kappa-n}}_n(\phi(u_q-it))\, dt
-\int_{-\infty}^{\infty}\frac{1}{u_q-it} \, f^{q\mu^{\kappa-n}}_n(\phi(u_q-it))\, dt.
\label{dec}
\end{aligned}\end{equation}
The third and fourth integrals on the right hand side of~\eqref{dec} are negligible
by Lemmas~\ref{l_3} and~\ref{l_4} below, respectively. This is due to the fact that in the desired formula~\eqref{asymp}
$q\psi_n(\phi(u_q))\to 0$ uniformly in $q$, and $yu_q$ is positive and uniformly bounded away from zero.
We now show that the second integral is also
negligible, and that the first one has the required asymptotics.%
\medskip%


By definition of $\psi_n$, we have
\begin{align*}
f_n(z)=p_{\mu}^{-\frac{1}{\mu-1}}\exp\big\{\mu^n\big(b(z)-\psi_n(z)\big)\big\}
\end{align*}
and so
\begin{align}
\label{iter}
f^{q\mu^{\kappa-n}}_n(z)=p_{\mu}^{-\frac{q\mu^{\kappa-n}}{\mu-1}}\exp\big\{q\mu^{\kappa}\big(b(z)-\psi_n(z)\big)\big\}.
\end{align}
To estimate the second integral
on the right hand side of~\eqref{dec} we use~\eqref{ts} and~\eqref{ts2} with the uniform error bounds and get
\begin{align*}
\mathcal{R}e&\big[qb(\phi(u_q-it))-q\psi_n(\phi(u_q-it))+y(u_q-it)\big]\\
&\le qb(\phi(u_q))-q\psi_n(\phi(u_q))+yu_q-\frac{q\sigma_q^2t^2}{4}
\le qb(\phi(u_q))+yu_q-\frac{\sigma_q^2t^2}{4},
\end{align*}
for all $|t|\le \theta$, using that $\psi_n$ is positive. Hence
\begin{align*}
&\Big|\int_{|t|\in [\rho,\theta]}
\frac{e^{y \mu^{\kappa}(u_q-it)}}{u_q-it}f^{q\mu^{\kappa-n}}_n(\phi(u_q-it))\, dt\Big|\\
&=p_{\mu}^{-\frac{q\mu^{\kappa-n}}{\mu-1}}\Big|\int_{|t|\in [\rho,\theta]}
\frac{1}{u_q-it}\exp\big\{\mu^{\kappa}\big(qb(\phi(u_q-it))-q\psi_n(\phi(u_q-it))+y(u_q-it)\big)\big\}\, dt\Big|\\
&\le\frac{2\theta}{u_q}\, p_{\mu}^{-\frac{q\mu^{\kappa-n}}{\mu-1}}\exp\Big\{\mu^{\kappa}\Big(qb(\phi(u_q))+yu_q
-\frac{\sigma_q^2\rho^2}{4}\Big)\Big\}\\
&=o(1)\, p_{\mu}^{-\frac{q\mu^{\kappa-n}}{\mu-1}}\mu^{-\frac{\kappa}{2}}\exp\big\{\mu^{\kappa}\big(qb(\phi(u_q))+yu_q\big)\big\}
\end{align*}
uniformly in $q$ since
\begin{align*}
\exp\Big\{-\frac{\mu^\kappa\sigma_q^2\rho^2}{4}\Big\}=\exp\Big\{-\frac{\kappa^2\sigma_q^2}{4}\Big\}
=o(1)\,\mu^{-\frac{\kappa}{2}}.
\end{align*}
%
Now consider the first integral on the r.h.s. of~\eqref{dec}, which is the only one contributing to the asymptotics.
Using~\eqref{ts},~\eqref{ts2},~\eqref{iter}, and dropping the $O$ terms since $\mu^\kappa\rho^3\to 0$ we get
\begin{align*}
\int_{-\rho}^{\rho}
&\frac{e^{y \mu^{\kappa}(u_q-it)}}{u_q-it}
f^{q\mu^{\kappa-n}}_n(\phi(u_q-it))\, dt\\
&=p_{\mu}^{-\frac{q\mu^{\kappa-n}}{\mu-1}}\int_{-\rho}^{\rho}
\frac{1}{u_q-it}\exp\big\{\mu^{\kappa}\big(qb(\phi(u_q-it))-q\psi_n(\phi(u_q-it))+y(u_q-it)\big)\big\}\, dt\\
&=p_{\mu}^{-\frac{q\mu^{\kappa-n}}{\mu-1}}\frac{1+o(1)}{u_q}
\exp\big\{\mu^{\kappa}\big(qb(\phi(u_q))-q\psi_n(\phi(u_q))+yu_q\big)\big\}\\
&\qquad \times\int_{-\rho}^{\rho}
\exp\Big\{itq\mu^\kappa a_q-\frac{q\mu^{\kappa}(\sigma_q^2-s_q)t^2}{2}\Big\}\, dt.
\end{align*}
Using the substitution $\tau=t\mu^{\frac{\kappa}{2}}\sqrt{(\sigma_q^2-s_q)q}$, we obtain
\begin{align*}
\int_{-\rho}^{\rho}
&\frac{e^{y \mu^{\kappa}(u_q-it)}}{u_q-it}
f^{q\mu^{\kappa-n}}_n(\phi(u_q-it))\, dt\\
&=p_{\mu}^{-\frac{q\mu^{\kappa-n}}{\mu-1}}\frac{\sqrt{2\pi}}{\sigma_q u_q \sqrt{q}}
\mu^{-\frac{\kappa}{2}}\exp\big\{\mu^{\kappa}\big(qb(\phi(u_q))-q\psi_n(\phi(u_q))+yu_q\big)\big\}I_q,
\end{align*}
where
\begin{align*}
I_q(\e)=\frac{1+o(1)}{\sqrt{2\pi}}\int_{-\kappa\sqrt{(\sigma^2_q-s_q)q}}^{\kappa\sqrt{(\sigma^2_q-s_q)q}}
\exp\Big\{i\tau \mu^{\frac{\kappa}{2}}a_q\sqrt{\frac{q}{\sigma_q^2-s_q}}-\frac{\tau^2}{2}\Big\}\, d\tau.
\end{align*}
%
It is easy to see that the absolute value of the integral on the right hand side is bounded by $\sqrt{2\pi}$.
Since it is clearly nonnegative, we get the uniform bound $I_q\le 1+o(1)$.
If $\mu^{\kappa}\psi_n(\phi(u_1))=O(1)$ then
$I_1\to 1$ by Lemma~\ref{normal} with
\begin{align*}
\rho_1(\e) & :=\kappa\sqrt{\sigma^2_1-s_1}\to \infty\\
\rho_2(\e) & :=\mu^{\frac{\kappa}{2}}a_1\sqrt{\frac{1}{\sigma_1^2-s_1}}=
O(1)\,\mu^{\frac \kappa 2+n}\psi_n(\phi(u_1))\to 0,
\end{align*}
where the last line follows from Lemma~\ref{l_psi}.
\end{proof}
\medskip

\begin{lemma}
\label{l_4}
There is $c>0$ such that
\begin{align*}
\Big|\int_{-\infty}^{\infty}\frac{1}{u_q-it}\, f^{q\mu^{\kappa-n}}_n(\phi(u_q-it))\, dt\Big|\le cp_{\mu}^{-\frac{q\mu^{\kappa-n}}{\mu-1}}\exp\big\{\mu^{\kappa}qb(\phi(u_q))\big\},
\end{align*}
for any  $q\in [1,2]$ such that $q\mu^{\kappa-n}\in \N$ and any $\e>0$.
\end{lemma}

\begin{proof}
Observe that
$f_n^{q\mu^{\kappa-n}}(z)/z$ is a series with non-negative coefficients and so an increasing function on $(0,1)$.
Since $|\phi(u_q-it)|\le \phi(u_q)\le \phi(u_*)$ we have
\begin{align*}
\Big|\int_{-\infty}^{\infty}\frac{1}{u_q-it}\,f^{q\mu^{\kappa-n}}_n(\phi(u_q-it))\, dt\Big|
&\le \frac{1}{u_q}\int_{-\infty}^{\infty}f^{q\mu^{\kappa-n}}_n(|\phi(u_q-it)|)\,dt\\
&=\frac{1}{u_q}\int_{-\infty}^{\infty}\frac{f^{q\mu^{\kappa-n}}_n(|\phi(u_q-it)|)}{|\phi(u_q-it)|}|\phi(u_q-it)| \, dt\\
&\le \frac{f^{q\mu^{\kappa-n}}_n(\phi(u_q))}{u_*\phi(u_*)}\int_{-\infty}^{\infty}|\phi(u_q-it)| \, dt.
\end{align*}
The integral is uniformly bounded by Lemma 16 in~\cite{FW09}. Lemma 13 from the same paper implies the estimate
$$f^{q\mu^{\kappa-n}}_n(\phi(u_q))<p_{\mu}^{-\frac{q\mu^{\kappa-n}}{\mu-1}}\exp\big\{\mu^\kappa qb(\phi(u_q))\big\},$$
which completes  the proof.
\end{proof}
\medskip

\begin{lemma}
\label{l_3}
For any $\theta>0$ there are $\delta>0$ and
$c>0$ such that
\begin{align*}
\Big|\int_{|t|\ge \theta}\frac{e^{y \mu^{\kappa}(u_q-it)}}{u_q-it}f_n^{q\mu^{\kappa-n}}(\phi(u_q-it))\, dt\Big|
\le c p_{\mu}^{-\frac{q\mu^{\kappa-n}}{\mu-1}}\exp\big\{\mu^{\kappa}\big(qb(\phi(u_q))+yu_q-\delta\big)\big\}
\end{align*}
for any $q\in [1,2]$ such that $q\mu^{\kappa-n}\in \N$ and any $\e$.
\end{lemma}

\begin{proof} Following the proof of Lemma 16 in~\cite{FW09}, we use the fact that, for each $w\in [u_*,u^*]$,
$t\mapsto \phi(w-it)/\phi(w)$ is the characteristic function of some absolutely continuous law
(Cram\'er transform), the continuity of the mapping $(w,t)\mapsto \phi(w-it)/\phi(w)$, and the compactness of $[u_*,u^*]$
to conclude that there is a constant $\eta$ such that
\begin{align*}
|\phi(u_q-it)|<(1-\eta)\phi(u_q)\qquad \text{ for all } |t|>\theta.
\end{align*}
Using the monotonicity of $z\mapsto f_n^{q\mu^{\kappa-n}}(z)/z$ on~$(0,1)$, we get
\begin{align*}
|f^{q\mu^{\kappa-n}}_n(\phi(u_q-it))|
&\le f^{q\mu^{\kappa-n}}_n(|\phi(u_q-it)|)\\
&=\frac{f^{q\mu^{\kappa-n}}_n(|\phi(u_q-it)|)}{|\phi(u_q-it)|}|\phi(u_q-it)|
\le \frac{f^{q\mu^{\kappa-n}}_n((1-\eta)\phi(u_q))}{(1-\eta)\phi(u_q)}|\phi(u_q-it)|.
\end{align*}
Using Lemmas 13 and 16 from~\cite{FW09} we obtain, for some $c>0$,
\begin{align*}
\Big|\int_{|t|\ge \theta}\frac{e^{y \mu^{\kappa}(u_q-it)}}{u_q-it}f_n^{q\mu^{\kappa-n}}(\phi(u_q-it))\, dt\Big|
&\le e^{u_qy\mu^{\kappa}}
\frac{f^{q\mu^{\kappa-n}}_n((1-\eta)\phi(u_q))}{(1-\eta)u_*\phi(u_*)}\int_{-\infty}^\infty|\phi(u_q-it)|\, dt\\
&\le c p_{\mu}^{-\frac{q\mu^{\kappa-n}}{\mu-1}}\exp\big\{\mu^{\kappa}\big(qb((1-\eta)\phi(u_q))+u_qy\big)\big\}.\\
\end{align*}
By Lemma 14 in~\cite{FW09} we have $b'(s)\ge 1/s> 1$ on $(0,1)$. Hence
$b(\phi(u_q))-b((1-\eta)\phi(u_q))\ge \eta\phi(u_q)\ge \eta\phi(u^*).$
Picking $\delta=\eta\phi(u_2)$ we obtain the desired estimate since $\phi(u_q)\ge \phi(u_2)$.
\end{proof}
\medskip

\begin{lemma} If $\rho_1\to \infty$ and $\rho_2\to 0$ then
\label{normal}
\begin{align*}
\int_{-\rho_1}^{\rho_1}\exp\Big\{i\tau\rho_2-\frac{\tau^2}{2}\Big\}\, d\tau =\sqrt{2\pi}+o(1).
\end{align*}
\end{lemma}

\begin{proof}
We have
\begin{align*}
\int_{-\rho_1}^{\rho_1}\exp\Big\{i\tau\rho_2-\frac{\tau^2}{2}\Big\}\, d\tau
=e^{-\rho_2^2/2}\int_{-\rho_1}^{\rho_1}\exp\Big\{-\frac{(\tau-i\rho_2)^2}{2}\Big\}\, d\tau.
\end{align*}
Denote by $\Gamma^1(\e)$ the straight path in $\C$ going from $-\rho_1-i\rho_2$ to
$-\rho_1$ and by $\Gamma^2(\e)$ the straight path in $\C$ going from $\rho_1$ to
$\rho_1-i\rho_2$.
Since $z\mapsto \exp\{-z^2/2\}$ is an entire function we have
\begin{align*}
\int_{-\rho_1}^{\rho_1}\exp\Big\{-\frac{(\tau-i\rho_2)^2}{2}\Big\}\, d\tau
=\int_{-\rho_1}^{\rho_1}e^{-\tau^2/2}\, d\tau
+\int_{\Gamma^1\cup \Gamma^2}\!\!\!\! e^{-z^2/2}dz.
\end{align*}

Obviously, the first integral converges to $\sqrt{2\pi}$. The second integral tends to zero since
the length of $\Gamma^1\cup \Gamma^2$ goes to zero and $|e^{-z^2/2}|\le e^{-(\rho_1^2-\rho_2^2)/2}\to 0$ on
$\Gamma^1\cup \Gamma^2$.
\end{proof}
\medskip


\section{Time of the first non-minimal branching}

In this section we prove Theorems~\ref{t_main} and~\ref{t_main2}. The key idea, just as in the case~$\mu=1$, is to
combine a decomposition of the population according to their ancestry in a suitably chosen generation with the tail
estimate for sums of independent copies of~$W$.

\begin{lemma}
\label{01}
Fix $d \in \{-1, 0,1\}$.
\smallskip

(a) If $\e_j\downarrow 0$ such that $\mu^{\kappa(\e_j)}\psi_{n(\e_j)}(\phi(u_1(\e_j)))\to 0$, then
\begin{align*}
\P\big(\, {K}>\kappa(\e_j)-n(\e_j)\, \big| \, W<\e_j\big)\to 1.
\end{align*}
\smallskip

(b) If $\e_j\downarrow  0$ such that
$\mu^{\kappa(\e_j)}\psi_{n(\e_j)}(\phi(u_1(\e_j))) \asymp1$, then
\begin{align*}
0 \prec \P\big(\, {K}>\kappa(\eps_j)-n(\eps_j)\, \big| \, W<\e_j\big)\prec 1.
\end{align*}
\smallskip

(c) If $\e_j\downarrow  0$ such that $\mu^{\kappa(\eps_j)}\psi_{n(\e_j)}(\phi(u_1(\eps_j)))\to \infty$, then
\begin{align*}
\P\big(\, {K}>\kappa(\eps_j)-n(\eps_j)\, \big| \, W<\e_j\big)\to 0.
\end{align*}
\end{lemma}

(Recall that $0\prec q_j \prec 1$ means that the sequence $q_j$ is uniformly bounded away from 0 and~1).

\begin{proof} Decomposing the tree according to ancestry in generation $\kappa-n$, we get
\begin{align*}
\P\big({K}> & \,\kappa(\eps)-n(\eps), W<\e\big) \\ & =\P(Z_{\kappa-n}=\mu^{\kappa-n}, W<\e)
 =\P\Big(Z_{\kappa-n}=\mu^{\kappa-n}, \sum_{i=1}^{\mu^{\kappa-n}}W_i<\e a^{\kappa-n}\Big)\\
&=\P(Z_{\kappa-n}=\mu^{\kappa-n})\P\Big(\sum_{i=1}^{\mu^{\kappa-n}}W_i<\e a^{\kappa-n}\Big)
=p_{\mu}^{\frac{\mu^{\kappa-n}-1}{\mu-1}}\P\Big(\sum_{i=1}^{\mu^{\kappa-n}}W_i<\e a^{\kappa-n}\Big).
\end{align*}
Hence, combining Proposition~\ref{p_main} with $q=1$ and Lemma~\ref{eps}, we obtain
\begin{align*}
\P\big(\, {K}>\kappa(\eps)-n(\eps)\, \big| \, W<\e\big)
&=p_{\mu}^{\frac{\mu^{\kappa-n}-1}{\mu-1}}\P\Big(\sum_{i=1}^{\mu^{\kappa-n}}W_i<\e a^{\kappa-n}\Big)\P(W<\e)^{-1}\\
&=\exp\big\{-\mu^{\kappa(\eps)}\psi_{n(\eps)}(\phi(u_1(\eps)))+o(1)\big\}\,I_1(\eps).
\end{align*}
In case~(a) and (b) we have $I_1(\eps_j)\to 1$ and the result follows. In case~(c) we use
that  $I_1(\eps_j)$ is bounded and therefore $\P\big(\, {K}>\kappa(\eps_j)-n(\eps_j)\ | \ W<\e_j\big)\to 0$.
\end{proof}
\medskip

It remains to analyse the expression  $\mu^{\kappa}\psi_n(\phi(u_1))$ for different sequences $\eps_n\downarrow 0$. We
prepare this by collecting three auxiliary facts.

\begin{lemma}
\label{l_test}
As $\e\downarrow 0$ we have\\[-2mm]
\begin{itemize}
\item[(a)] $\displaystyle
\mu^{\kappa-n}\asymp  \frac{\e^{-\alpha}}{\log (1/\e)},$\\[-2mm]
\item[(b)] $\displaystyle\exp\big\{(\lambda-\mu)\mu^nb(\phi(u_1))\big\}=\e^{\alpha\mu^{-\{\gamma\}-d}},$\\[-2mm]
\item[(c)] $\displaystyle
\mu^{\kappa}\psi_n(\phi(u_1)) \asymp\frac{\e^{\alpha(\mu^{-\{\gamma\}-d}-1)}}{\log (1/\e)}.$
\end{itemize}
\end{lemma}
(Recall that $\{\gamma\} = \lceil \gamma \rceil - \gamma$.)

\begin{proof}
Observe that it follows from $a=\mu^{1/\beta}$ and the definition of $y$ that
$\mu^\kappa=(y/\e)^{\alpha}$.
By definition of $\gamma$ and~$n$ we have
$\mu^{n-\kappa} \asymp \mu^{-\gamma}=\e^{\alpha}\log (1/\e)\,\mu^{-H}$,
which implies~(a). By the definition of $H$, see (\ref{Hdef}), we have
$\mu^H=-b(\phi(u_1))y^{\alpha}\alpha^{-1}(\lambda-\mu)$.
Combining these facts we obtain
\begin{align}
(\lambda-\mu)\mu^n b(\phi(u_{1}))
&=(\lambda-\mu)\mu^{\kappa-\gamma-\{\gamma\}-d}b(\phi(u_1))\notag\\
&=(\lambda-\mu)\log(1/\e)y^{\alpha}\mu^{-H-\{\gamma\}-d}b(\phi(u_1))
=\alpha\mu^{-\{\gamma\}-d}\log\e,
\label{ss4}
\end{align}
which proves~(b). By Lemma~\ref{l_psi}, part~(a) and~\eqref{ss4} we have
\begin{align*}
\mu^{\kappa}\psi_n(\phi(u_1))
&\asymp\mu^{\kappa-n-1}\exp\big\{(\lambda-\mu)\mu^nb(\phi(u_1))\big\}
\asymp\frac{\e^{\alpha(\mu^{-\{\gamma\}-d}-1)}}{\log (1/\e)},
\end{align*}
proving~(c).
\end{proof}
\bigskip

\begin{proof}[Proof of Theorem~\ref{t_main}]
Let $d=-1$, so that $\kappa-n=\lceil\gamma\rceil -1$.
By Lemma~\ref{l_test}\,(c), $\mu^{\kappa}\psi_n(\phi(u_1))\to 0$ since $\mu^{-\{\gamma\}+1}-1\ge 0$.
Hence Lemma~\ref{01} implies
\begin{align}
\label{tt1}
\P\big(\, {K}>\lceil\gamma\rceil -1\ | \ W<\e\big)
=\P\big(\, {K}>\kappa-n\ | \ W<\e\big)
\to 1\qquad \text{ as }\e\to 0.
\end{align}
Now let $d=1$, so that $\kappa-n=\lceil\gamma\rceil +1$.
Again, by Lemma~\ref{l_test}\,(c), $\mu^{\kappa}\psi_n(\phi(u_1))\to \infty$ as now
$\mu^{-\{\gamma\}-1}-1<0$. Hence Lemma~\ref{01} implies
\begin{align}
\label{tt2}
\P\big(\, {K}>\lceil\gamma\rceil +1\ | \ W<\e\big)
=\P\big(\, {K}>\kappa-n\ | \ W<\e\big)
\to 0\qquad \text{ as }\e\to 0.
\end{align}
The statement now follows from~\eqref{tt1} and~\eqref{tt2}.
\end{proof}
\medskip

\begin{proof}[Proof of Theorem~\ref{t_main2}]
Let $d=0$.
Then $\kappa-n=\lceil\gamma\rceil $.
By Lemma~\ref{l_test}\,(c) we have $$\mu^{\kappa}\psi_n(\phi(u_1))\asymp \omega.$$%
%
In case~(a) of Theorem \ref{t_main2}, we have $\omega(\e_j)\to \infty$ and so $\mu^{\kappa}\psi_n(\phi(u_1))\to \infty$
by Lemma~\ref{l_test}. Hence Lemma~\ref{01} implies
$\P({K}>\lceil\gamma(\eps_j)\rceil\ | \ W<\e_j)
=\P({K}>\kappa(\eps_j)-n(\eps_j)\, | \, W<\e_j)\to 0.$
Together with Theorem~\ref{t_main} we get
$\P\big(\, {K}=\lceil\gamma(\eps_j)\rceil\ | \ W<\e_j\big)\to 1$.
\medskip

In case~(b) we have $\omega(\e_j)\asymp 1$ and so $\mu^{\kappa}\psi_n(\phi(u_1))\asymp 1$ by Lemma~\ref{l_test}.
Hence Lemma~\ref{01} implies that
$\P(K>\lceil\gamma(\eps_j)\rceil\,| \, W<\e_j) =\P(K>\kappa(\eps_j)-n(\eps_j)\, | \, W<\e_j)$
is asymptotically equivalent to $\exp\{-\mu^{\kappa}\psi_{n}(\phi(u_1))\}.$
Together with Theorem~\ref{t_main} we infer that
$0 \prec \P( K=\lceil\gamma(\eps_j)\rceil\, | \, W<\e_j)\prec 1$ and
$0\prec \P( K=\lceil\gamma(\eps_j)\rceil+1\, | \, W<\e_j)\prec 1$, as required.
\medskip

In case~(c)  $\omega(\e_j)\to 0$ and so $\mu^{\kappa}\psi_n(\phi(u_1))\to 0$
by Lemma~\ref{l_test}. Hence Lemma~\ref{01} implies
$\P(K>\lceil\gamma(\eps_j)\rceil\, | \, W<\e_j)
=\P(K>\kappa(\eps_j)-n(\eps_j)\, | \, W<\e_j)\to 1.$
Together with Theorem~\ref{t_main} this implies the statement.
\end{proof}


\section{Extra offspring in the critical generation}

In this section we prove Theorem~\ref{t_prop}. Denote $\J:=\{j\ge \lambda: p_j\neq 0\}$ and
\begin{align*}
\M:=\big\{ (m_j)_{j\in \J}: m_j\in \N\cup\{0\} \text{ for all }j\in \J\big\}.
\end{align*}
For each $m\in \M$, denote
\begin{align*}
|m|:=\sum_{j\in \J}m_j \qquad\text{and}\qquad \langle m\rangle:=\sum_{j\in \J}(j-\mu)m_j\in \N\cup \{0,\infty\}.
\end{align*}
For each $j\in \J$, denote by $M_j$ the number of individuals in generation $K-1$ having precisely~$j$
children and let $M:=(M_j)_{j\in \J}$. The strategy of the proof is as follows. We first show that
$Z_{K}=\mu^{K}+\langle M\rangle$, see (\ref{Zmudiff}).
We then prove that, conditioned on the event $ W <\e$, the random variable 
$\langle M\rangle$ is, with high probability, in a certain interval,
see (\ref{boundsMph}).
Not surprisingly, in order to show (\ref{boundsMph}), we have to give the asymptotic behaviour of
$\P(M=m, K=\kappa-n \, |\,  W <\e)$, see (\ref{terms}), resulting in (\ref{ttt}), which has to be
optimized over $m$.\\

For each $t>0$, denote
\begin{align*}
\M_t:=\big\{m\in\M: \langle m \rangle <t\big\}.
\end{align*}

\begin{lemma}
\label{card}
The cardinality of $\M_t$ satisfies $|\M_t|=e^{o(t)}$ as $t\to\infty$.
\end{lemma}

\begin{proof} For each $n\in \N$ and $t>0$,  denote
\begin{align*}
S_{n,t}=\{m\in (\N\cup\{0\})^{n}:\sum_{i=1}^n m_i<t\}.
\end{align*}
Let $Q_n = [0,1]^n$ be the unit $n$-dimensional cube based in the origin. Then
\begin{align}
\label{sim1}
|S_{n,t}|=\text{vol}\Big\{\bigcup_{m\in S_{n,t}}(m+Q)\Big\}
\le \text{vol}\Big\{x\in [0,\infty)^n: \sum_{i=1}^n x_i<t+n\Big\}=\frac{(t+n)^n}{n!}.
\end{align}
On the other hand,
\begin{align}
\label{sim2}
|S_{n,t}|=\sum_{0\le j<t}\Big|\Big\{m\in (\N\cup\{0\})^{n}:\sum_{i=1}^n m_i=j\Big\}\Big|\le \sum_{0\le j<t} n^j\le \int_0^{t}n^xdx\le n^{t+1}.
\end{align}
The former estimate is useful for large $t$, the latter for large $n$.
\medskip

Let $r\colon(0,\infty)\to \N$ be such that
$r_t=o(t/\log t)$ and $\log t=o(r_t)$ as $t\to \infty$. With the convention
$m_i=0$ if $i\not\in\J$ we get, for large $t$,
\begin{align*}
\M_t\subset \big\{m\in (\N\cup\{0\})^{\N} \colon (\lambda-\mu)\sum_{^{i=1}}^{_{r_t}}m_i<t,
(r_t-\mu)\sum_{^{i=r_t+1}}^{_{\lfloor t+\mu \rfloor}}m_i<t, m_i=0\text{ for all }i>t+\mu\big\}.
\end{align*}
Using~\eqref{sim1} and~\eqref{sim2} we get
\begin{align*}
|\M_t|
\le |S_{r_t,\frac{t}{\lambda-\mu}}|\,|S_{\lfloor t+\mu \rfloor-r_t, \frac{t}{r_t-\mu}}|
\le |S_{r_t,t}|\,|S_{\lfloor t+\mu \rfloor-r_t, \frac{t}{r_t-\mu}}|
\le \frac{(t+r_t)^{r_t}}{r_t!}(t+\mu-r_t)^{\frac{t}{r_t-\mu}+1}.
\end{align*}
This leads to
\begin{align*}
|\M_t|
&=\exp\Big\{r_t\log(t+r_t)-r_t\log r_t+r_t+\sfrac{t+\mu}{r_t-\mu}\log(t-r_t)+o(t)\Big\}\\
&=\exp\Big\{r_t\log t-r_t\log r_t+r_t+\sfrac{t+\mu}{r_t-\mu}\log t+o(t)\Big\}=e^{o(t)}.
\end{align*}
\end{proof}

\begin{lemma}
\label{taylor}
For $q\in[1,2]$ and $\e>0$ let
$h(q) :=
q(b\circ \phi)(u_q)+yu_q.$
Then
$$h(q)\le h(1)+(b\circ \phi)(u_1)(q-1).$$
\end{lemma}

\begin{proof} Since $b\circ \phi$ is analytic we get, using~\eqref{def_u},
\begin{align*}
\frac{\partial h}{\partial q}(q)=(b\circ \phi)(u_q)\qquad\text{ implying }\qquad
\frac{\partial h}{\partial q}(1)=(b\circ \phi)(u_1)
\end{align*}
and
\begin{align*}
\frac{\partial^2 h}{\partial q^2}(q)=(b\circ \phi)'(u_q)\frac{\partial u_q}{\partial q}.
\end{align*}
Since $(b\circ \phi)'$ is analytic and increasing
from $-\infty$ to $0$ on $(0,\infty)$, equation~\eqref{def_u} implies that
$u_q$ is increasing in $q$ and so the derivative ${\partial u_q}/{\partial q}$ is nonnegative.
Since $b\circ \phi$ is negative we have ${\partial^2 h}/{\partial q^2}(q)\le 0$ for all $q$ and $\e$.
Now the statement of the lemma follows from the Taylor expansion of $h$ at the point $q=1$.
\end{proof}
\medskip

Denote
\begin{equation}\label{Ndef}
N(\e):=\mu^{\kappa(\eps)-n(\eps)-1}
\end{equation}
 and let
\begin{align*}
\Phi_j(\e):=
p_jp_{\mu}^{-\frac{j-1}{\mu-1}}N\exp\big\{(j-\mu)\mu^n b(\phi(u_{1}))\big\}, \qquad \mbox{ for }  j\in \J.
\end{align*}
Note that all $\Phi_j$, $j\neq \lambda$, are negligible with respect to $\Phi_{\lambda}$ to the extent that,
for any~$c\in \R$,
\begin{align*}
\sum_{j\in \J\backslash \{\lambda\}}\Phi_j e^{cj}=o(\Phi_{\lambda}).
\end{align*}
\smallskip

\begin{proof}[Proof of Theorem~\ref{t_prop}]
Recall that for each $j\in \J$, $M_j$ is the number of individuals in generation $K-1$ having precisely~$j$
children. Write $M:=(M_j)_{j\in \J}$.  Then
\begin{equation}
\label{Zmudiff}
Z_{K}=\mu(\mu^{K-1}-|M|)+\sum_{j\in \J}j M_j=\mu^{K}+\langle M\rangle.
\end{equation}
Observe that, by Lemma~\ref{l_test}\,(b), we have
$(\lambda-\mu)\Phi_{\lambda}=C\,\mu^{\kappa-n}\e^{\alpha\mu^{\gamma-\kappa+n}},$
where
$$C=\big(\sfrac{\lambda}{\mu}-1\big)p_{\lambda}p_{\mu}^{-\frac{\lambda-1}{\mu-1}}.$$
Let $\delta>0$ be small enough. By Theorems~\ref{t_main} and~\ref{t_main2}
it suffices to show that
\begin{equation}
\label{boundsMph}
\P\big(\langle M\rangle\notin \big((\lambda-\mu-\delta)\Phi_{\lambda}(\eps_j),(\lambda-\mu+\delta)\Phi_{\lambda}(\eps_j)\big),
K=\kappa(\eps_j)-n(\eps_j)\,  \big|\,  W<\e_j\big)\to 0,
\end{equation}
for $d=0$ in the case $\omega(\e_j)\to \infty$  and for $d=1$ in the case $\omega(\e_j)\to 0$.
By Lemma~\ref{l_test},
\begin{align}
\label{ph1}
&\Phi_{\lambda}(\e_j)\asymp \frac{\e_j^{\alpha(\mu^{-\{\gamma\}}-1)}}{\log(1/\e_j)}= \omega(\e_j) \to\infty  &
\qquad \text{ for $d=0$ in the case $\omega(\e_j){\to} \infty$}, \\
&\Phi_{\lambda}(\e_j)\asymp \frac{\e_j^{\alpha(\mu^{-\{\gamma\}-1}-1)}}{\log(1/\e_j)} \to\infty &
\qquad \text{ for $d=1$ in the case $\omega(\e_j){\to} 0$}.
\label{ph2}
\end{align}
Hence in both cases $\Phi_{\lambda}(\e_j)\to\infty$.
\medskip

We prove (\ref{boundsMph}) by showing that
\begin{align}
\label{delta1}
\P\big(\langle M\rangle\le(\lambda-\mu-\delta)\Phi_{\lambda}(\e_j), K=\kappa(\e_j)-n(\e_j)\ \big|\ W<\e_j\big)\to 0,\\
\label{delta2}
\P\big((\lambda-\mu+\delta)\Phi_{\lambda}(\e_j)\le \langle M\rangle <3e^{\lambda}\Phi_{\lambda}(\e_j), K=\kappa(\e_j)-n(\e_j)\ \big|\ W<\e_j\big)\to 0,\\
\label{m1}
\P\big( 3e^{\lambda}\Phi_{\lambda}(\e_j)\le \langle M\rangle<N(\e_j)/2, K=\kappa(\e_j)-n(\e_j)\ \big|\ W<\e_j\big)\to 0,\\
\label{m2}
\P\big( \langle M\rangle\ge N(\e_j)/2, K=\kappa(\e_j)-n(\e_j)\ \big|\ W<\e_j\big)\to 0.
\end{align}

The rest of the proof is split into five steps. In Step~1, we find an asymptotic formula for the conditional
probabilities $\P(M=m, K=\kappa-n\, |\, W<\e)$ for $m\in \M$.
Then we prove~\eqref{delta1}, \eqref{delta2}, \eqref{m1}, and~\eqref{m2} in the next four steps.
\bigskip

{\it Step 1. An asymptotic formula.}

Let $m\in \M$ be such that $m_j\neq 0$ for some $j\in \J$ and $|m|\le N$.
In particular, this means that only finitely many of the $m_j$  are non-zero. Denote
$$q(m)
:=1+\langle m\rangle\mu^{n-\kappa}
\qquad\text{ and }\qquad \bar q(m)
:=2\wedge q(m),$$
where $\wedge$ stands for the minimum.
\medskip

For each $j\in \J$, denote by $\widetilde M_j(\e)$ the number of individuals in generation $\kappa-n-1$ having precisely~$j$ children.
Let $\widetilde M(\e):=(\widetilde M_j(\e))_{j\in \J}$. Again we drop the dependence on $\eps$ from this notation whenever convenient.
Observe that  $K=\kappa-n$ and $M=m$ imply $Z_{\kappa-n}=q(m)\mu^{\kappa-n}$ and so we have
\begin{align*}
\P(M=m,\, K=\kappa-n, W<\e)
=\P\Big(Z_{\kappa-n-1}=N, \widetilde M=m, \sum_{^{i=1}}^{_{q(m)\mu^{\kappa-n}}}W_i<\e a^{\kappa-n}\Big) & \notag\\
=\P\big(Z_{\kappa-n-1}=N\big)\,\P\big(\widetilde M=m \, \big|\,  Z_{\kappa-n-1}=N\big) \,
\P\Big(\sum_{^{i=1}}^{_{\bar q(m)\mu^{\kappa-n}}}W_i<\e a^{\kappa-n}\Big)&.
\end{align*}
This yields
\begin{align}
\P&\big(M=m, K=\kappa-n \, |\,  W <\e\big)\notag\\
&=\P\big(Z_{\kappa-n-1}=N\big)\P\big( \widetilde M=m\, \big| \, Z_{\kappa-n-1}=N\big) \,
\P\Big(\sum_{^{i=1}}^{_{\bar q(m)\mu^{\kappa-n}}}W_i<\e a^{\kappa-n}\Big)\,
\P(W<\e)^{-1}.
\label{terms}
\end{align}

For the \emph{first} term in (\ref{terms}), we have
\begin{align}
\label{term1}
\P\big(Z_{\kappa-n-1}=N\big)=p_{\mu}^{1+\mu+\cdots+\mu^{\kappa-n-2}}
=p_{\mu}^{\frac{N-1}{\mu-1}}.
\end{align}

We can compute the \emph{second} term in (\ref{terms}) as
\begin{align*}
P&\big( \widetilde M=m\, \big|\, Z_{\kappa-n-1}=N\big)
=p_{\mu}^{N}\frac{N!}{(N-|m|)!}\prod_{j\in \J}\frac{1}{m_j!}\Big(\frac{p_j}{p_{\mu}}\Big)^{m_j}.
\end{align*}
Observe that
\begin{align*}
\frac{\sqrt{N}}{\sqrt{N-|m|}}\prod_{\heap{j\in \J}{m_j\neq 0}}\frac{1}{\sqrt{m_j}}\le
\left\{\begin{array}{ll}
\sqrt{2} & \text{ if } |m|<N/2,\\
\sqrt{N} & \text{ if } |m|<N.
\end{array}\right.
\end{align*}
Combining this with Stirling's formula we obtain, uniformly for $|m|<N/2$,
\begin{align}
P&\big( \widetilde M=m\, \big|\, Z_{\kappa-n-1}=N\big)\notag\\
&=O(1)\, p_{\mu}^N\exp\Big\{N\log N-(N-|m|)\log (N-|m|)-\sum_{j\in \J}m_j\log m_j
+\sum_{j\in \J}m_j\log \frac{p_j}{p_{\mu}}\Big\}\notag\\
&=O(1)\, p_{\mu}^N\exp\Big\{|m|\log N-N\big(1-\sfrac{|m|}{N}\big)\log \big(1-\sfrac{|m|}{N}\big)
-\sum_{j\in \J}m_j\log m_j+\sum_{j\in \J}m_j\log \frac{p_j}{p_{\mu}}\Big\}\notag\\
&=O(1)\, p_{\mu}^N\exp\Big\{|m|\log N+|m|
-\sum_{j\in \J}m_j\log m_j+\sum_{j\in \J}m_j\log \frac{p_j}{p_{\mu}}\Big\},
\label{term21}
\end{align}
since $(1-x)\log (1-x)\ge -x$ for all $0\le x\le 1$ (we use the convention $0\log 0=0$).
%
Similarly,
\begin{align}
P&\big( \widetilde M=m\, \big|\, Z_{\kappa-n-1}=N\big)\notag\\
&=O(1) \, \sqrt{N}p_{\mu}^N\exp\Big\{|m|\log N+|m|
-\sum_{j\in \J}m_j\log m_j+\sum_{j\in \J}m_j\log \frac{p_j}{p_{\mu}}\Big\}
\label{term22}
\end{align}
uniformly for all $|m|\le N$.
\medskip

To compute the \emph{third} term in (\ref{terms}) we use Proposition~\ref{p_main} and get
\begin{align*}
\P\Big(\sum_{^{i=1}}^{_{\bar q(m)\mu^{\kappa-n}}}W_i<\e a^{\kappa-n}\Big)
=O(1)\, p_{\mu}^{-\frac{\bar q(m)\mu N}{\mu-1}}\mu^{-\frac{\kappa}{2}}
\exp\big\{\mu^\kappa h(\bar q(m))-\mu^\kappa \bar q(m) \psi_n(\phi(u_{\bar q(m)}))\big\}.
\end{align*}
Applying Lemma~\ref{taylor} yields
\begin{align}
\label{term3}
\P&\Big(\sum_{^{i=1}}^{_{\bar q(m)\mu^{\kappa-n}}}W_i<\e a^{\kappa-n}\Big)\notag\\
&=O(1)\, p_{\mu}^{-\frac{\bar q(m)\mu N}{\mu-1}}\mu^{-\frac{\kappa}{2}}\exp\big\{\mu^\kappa h(1)+
\langle m\rangle\mu^n(b\circ \phi)(u_1)
-\mu^\kappa \bar q(m) \psi_n(\phi(u_{\bar q(m)}))\big\}.
\end{align}
The \emph{fourth} term in (\ref{terms}) is given by Lemma~\ref{eps},
\begin{align}
\label{term4}
\P(W<\e)^{-1}=O(1)p_{\mu}^{\frac{1}{\mu-1}}\mu^{\frac{\kappa}{2}}\exp\big\{-\mu^\kappa h(1)\big\}.
\end{align}
Combining~\eqref{terms},~\eqref{term1},~\eqref{term21},~\eqref{term3}, and~\eqref{term4} we obtain
\begin{align}
\P\big(M&=m, K=\kappa-n \, |\,  W <\e\big)\notag\\
&=O(1)\, p_{\mu}^{-\frac{\langle m\rangle}{\mu-1}}
\exp\Big\{|m|\log N+|m|
-\sum_{j\in \J}m_j\log m_j+\sum_{j\in \J}m_j\log \frac{p_j}{p_{\mu}}\notag\\
&\phantom{aaaaaaaaaaaaaaaaaa}+\langle m\rangle\mu^n(b\circ \phi)(u_1)-\mu^\kappa \bar q(m) \psi_n(\phi(u_{\bar q(m)}))\Big\}\notag\\
&=O(1)\, \exp\Big\{|m|-\sum_{j\in \J}m_j\log m_j+\sum_{j\in \J}m_j
\log \Big(p_jp_{\mu}^{-\frac{j-1}{\mu-1}}N\exp\big\{(j-\mu)\mu^n b(\phi(u_{1}))\big\}\Big)\notag\\
&\phantom{aaaaaaaaaaaaa}-\mu^\kappa \bar q(m)\psi_n(\phi(u_{\bar q(m)}))\Big\}\notag\\
&=O(1)\, \exp\Big\{|m|-\sum_{j\in \J}m_j\log m_j+\sum_{j\in \J}m_j\log \Phi_j-\mu^\kappa \bar q(m)\psi_n(\phi(u_{\bar q(m)}))\Big\},
\label{uu1}
\end{align}
uniformly in $m$ such that $|m|<N/2$.
\medskip

If the condition $|m|<N/2$ is not satisfied we need to replace~\eqref{term21} by the rougher estimate~\eqref{term22} in the previous
computation. This gives
\begin{align}
\P\big(M&=m, K=\kappa-n \, |\,  W <\e\big)\notag\\
&=O(1)\,\sqrt{N}\exp\Big\{|m|-\sum_{j\in \J}m_j\log m_j+\sum_{j\in \J}m_j\log \Phi_j
-\mu^\kappa \bar q(m)\psi_n(\phi(u_{\bar q(m)}))\Big\}\notag\\
&=O(1)\,\sqrt{N}\exp\Big\{|m|-\sum_{j\in \J}m_j\log m_j+\sum_{j\in \J}m_j\log \Phi_j\Big\}
\label{uu12}
\end{align}
uniformly for all $|m|\le N$, since the last term in the second line of~\eqref{uu12} is positive.
\bigskip

\pagebreak[3]

{\it Step 2. Proof of~\eqref{delta1}.}

Consider all $m$ such that $\langle m\rangle \le (\lambda-\mu-\delta)\Phi_{\lambda}$.
Observe that in this case, for~$\eps>0$ small enough, $q(m)\le 2$ and $|m|\le \langle m\rangle< N/2$.
By Lemma~\ref{l_psi} we have
\begin{equation}\begin{aligned}
\mu^\kappa q(m)\psi_n( & \phi(u_{q(m)}))\\
&=p_{\lambda}p_{\mu}^{-\frac{\lambda-1}{\mu-1}}N
\exp\big\{(\lambda-\mu)\mu^n b(\phi(u_{q(m)}))\big\}(1+o(1))\\
&=p_{\lambda}p_{\mu}^{-\frac{\lambda-1}{\mu-1}}N
\exp\big\{(\lambda-\mu)\mu^n b(\phi(u_{1}))+\langle m\rangle\mu^{2n-\kappa} O(1)\big\}(1+o(1))\\
&=\Phi_{\lambda}+o(\Phi_{\lambda}).
\label{uu2}
\end{aligned}\end{equation}
since $\mu^{\kappa-2n}\asymp \e^{-\alpha}(\log(1/\e))^{-2}$ and so $\Phi_{\lambda}\mu^{2n-\kappa}=o(1)$ according to~\eqref{ph1}
and~\eqref{ph2}.

Combining~\eqref{uu1} and~\eqref{uu2} we get, uniformly in $m$,
\begin{equation}\begin{aligned}
\label{ttt}
\P\big(M=m, & \, K=\kappa-n \, |\,  W <\e\big)\\
&=O(1)\exp\Big\{|m|-\sum_{j\in \J}m_j\log m_j+\sum_{j\in \J}m_j\log \Phi_j-\Phi_{\lambda}+o(\Phi_{\lambda})\Big\}.
\end{aligned}\end{equation}
It is easy to see that the function in the brackets achieves its maximum at $m$ given by $m_j=\Phi_j$. However, this
$m$ does not satisfy the condition $\langle m\rangle \le (\lambda-\mu-\delta)\Phi_{\lambda}$, and so the
maximum over the admissible domain is achieved  on the boundary $\langle m\rangle =(\lambda-\mu-\delta)\Phi_{\lambda}$.
Using Lagrange multipliers,
we obtain that the maximum is attained for $m_j=\Phi_j e^{c(j-\mu)}$ for some $c<0$ (depending on $\eps$) such that
\begin{align*}
\sum_{j\in \J}(j-\mu)\Phi_j e^{c(j-\mu)}=(\lambda-\mu-\delta)\Phi_{\lambda}.
\end{align*}
Since all $\Phi_j$ with $j\neq \lambda$ are negligible with respect to $\Phi_{\lambda}$
we have that $(\lambda-\mu)e^{c(\lambda-\mu)}\sim \lambda-\mu-\delta$ and so $c$ is separated from zero.
%
Substituting the maximiser into~\eqref{ttt} and neglecting all $\Phi_j$ with $j\neq \lambda$ we get
\begin{equation}\begin{aligned}
\label{tttt}
\P\big(M=m, & \, K=\kappa-n \, |\,  W <\e\big)\\
&= O(1) \exp\Big\{-\Phi_{\lambda}\big(1-e^{c(\lambda-\mu)}+c(\lambda-\mu)e^{c(\lambda-\mu)}\big)+o(\Phi_{\lambda})\Big\}.
\end{aligned}\end{equation}
Observe that the function $\rho(x)=1-e^{x}+xe^{x}$ is decreasing from $1$ to $0$ on $(-\infty,0]$. Since $c$ is negative and separated from zero we obtain
\begin{align*}
\P\big(M=m, K=\kappa-n \, |\,  W <\e\big)
= O(1) \exp\big\{-\theta\Phi_{\lambda}+o(\Phi_{\lambda})\big\}.
\end{align*}
with some $\theta>0$, uniformly in $m$.
Finally, by Lemma~\ref{card},
\begin{align*}
\P\big(\langle M\rangle \le (\lambda-\mu-\delta)\Phi_{\lambda}, K=\kappa-n \, |\,  W <\e\big)
&= O(1) \exp\big\{-\theta\Phi_{\lambda}+o(\Phi_{\lambda})\big\}|\M_{(\lambda-\mu-\delta)\Phi_{\lambda}}|\\
&=O(1) \exp\big\{-\theta\Phi_{\lambda}+o(\Phi_{\lambda})\big\}=o(1).
\end{align*}

{\it Step 3. Proof of~\eqref{delta2}.}

Now consider all $m$ such that  $(\lambda-\mu+\delta)\Phi_{\lambda}\le \langle m\rangle < 3e^{\lambda}\Phi_{\lambda}$.
The estimates $q(m)\le 2$ and $|m|<N/2$ as well as the
asymptotics~\eqref{uu2} and~\eqref{ttt} remain true in this case and, similarly to the previous step, the maximum of the
function in the brackets in~\eqref{ttt} over the region $\langle m\rangle\ge (\lambda-\mu+\delta)\Phi_{\lambda}$
is attained on the boundary
$\langle m\rangle=(\lambda-\mu+\delta)\Phi_{\lambda}$
at $m$ given by $m_j=\Phi_j e^{c(j-\mu)}$ for some $c>0$, depending on $\eps$ but bounded away from zero.
\pagebreak[2]

We use~\eqref{tttt}, which is true in this case as well, and the fact that
$\rho$ is increasing from $0$ to $\infty$ on $[0,\infty)$ to obtain
$\P(M=m, K=\kappa-n \, |\,  W <\e)
= O(1) \exp\{-\theta\Phi_{\lambda}+o(\Phi_{\lambda})\}$
with some $\theta>0$, uniformly in $m$.
Finally, by Lemma~\ref{card},
\begin{align*}
\P\big((\lambda-\mu+\delta)\Phi_{\lambda}\le \langle M\rangle < 3e^{\lambda} & \Phi_{\lambda} , K=\kappa-n \, |\,  W <\e\big)\\
&= O(1) \exp\big\{-\theta\Phi_{\lambda}+o(\Phi_{\lambda})\big\}|\M_{3e^{\lambda}\Phi_{\lambda}}|\\
&=O(1) \exp\big\{-\theta\Phi_{\lambda}+o(\Phi_{\lambda})\big\}=o(1).
\end{align*}

{\it Step 4. Proof of~\eqref{m1}.}

Here we consider all $m$ satisfying $3e^{\lambda}\Phi_{\lambda}\le \langle m\rangle< N/2$. Then
again $|m|<N/2$ and $q(m)\le 2$.
Since the last term in~\eqref{uu1} is positive we have
\begin{align}
\P&\big(M=m, K=\kappa-n \, |\,  W <\e\big)
=O(1)\,\exp\Big\{|m|-\sum_{j\in \J}m_j\log m_j+\sum_{j\in \J}m_j\log \Phi_j\Big\}\notag\\
&=O(1)\,\prod_{j\in \J}e^{-jm_j}\exp\Big\{\sum_{j\in \J}(j+1)m_j-\sum_{j\in \J}m_j\log m_j+\sum_{j\in \J}m_j\log \Phi_j\Big\}.
\label{ppp}
\end{align}
The maximum of the function in the brackets is achieved for $m_j=\Phi_je^j$, which does not satisfy the condition
$\langle m\rangle\ge 3e^{\lambda}\Phi_{\lambda}$.
It is easy to see that the maximum over the region $\langle m\rangle\ge 3e^{\lambda}\Phi_{\lambda}$ is achieved on the boundary
$\langle m\rangle= 3e^{\lambda}\Phi_{\lambda}$ for $m$ given by $m_j=\Phi_je^{j+c(j-\mu)}$, and $c>0$ is such that
\begin{align}
\label{sim}
\sum_{j\in \J}\Phi_je^{j+c(j-\mu)}=3e^{\lambda}\Phi_{\lambda}.
\end{align}
Substituting the maximiser into~\eqref{ppp} we obtain
\begin{align*}
\P\big(M=m, K=\kappa-n \, |\,  W <\e\big)
&=O(1)\,\prod_{j\in \J}e^{-jm_j}\exp\Big\{-\sum_{j\in \J}(c(j-\mu)-1)\Phi_je^{j+c(j-\mu)}\Big\}.
\end{align*}
Since $e^{\lambda+c(\lambda-\mu)}\sim 3e^{\lambda}$ according to~\eqref{sim}
we have $c(\lambda-\mu)-1\sim \log 3-1>0$ and so
\begin{align*}
\P\big(M=m, K=\kappa-n \, |\,  W <\e\big)
&=O(1)\,\exp\big\{-\theta\Phi_{\lambda}+o(\Phi_{\lambda})\big\}\prod_{j\in \J}e^{-jm_j}
\end{align*}
for some $\theta>0$, uniformly in $m$.
Hence
\begin{align*}
\P&\big(3e^{\lambda}\Phi_{\lambda}\le \langle M\rangle <N/2, K=\kappa-n \, |\,  W <\e\big)\notag\\
&=O(1)\,\exp\big\{-\theta\Phi_{\lambda}+o(\Phi_{\lambda})\big\}\sum_{m\in \M_{N/2}}\prod_{j\in \J}e^{-jm_j},
\end{align*}
and the right hand side is $o(1)$ as
\begin{align}
\label{finn}
\sum_{m\in \M_{N/2}}\prod_{j\in \J}e^{-jm_j}\le \prod_{j\in \J}\sum_{m_j=0}^{\infty}e^{-jm_j}=\prod_{j\in \J}\frac{1}{1-e^{-j}}<\infty.
\end{align}
\bigskip

\pagebreak[3]
{\it Step 5. Proof of~\eqref{m2}.}

Finally consider all $m$ such that $ \langle m\rangle\ge N/2$. Using~\eqref{uu12} we obtain
\begin{align}
\P&\big(M=m, K=\kappa-n \, |\,  W <\e\big)\notag\\
&=O(1)\sqrt{N}\prod_{j\in \J}e^{-jm_j}\exp\Big\{\sum_{j\in \J}(j+1)m_j-\sum_{j\in \J}m_j\log m_j+\sum_{j\in \J}m_j\log \Phi_j\Big\}.
\label{ppp3}
\end{align}
Similarly to the previous step,
the maximum of the function in the brackets  over the region $\langle m\rangle\ge N/2$ is achieved on the boundary
$\langle m\rangle= N/2$ at $m$ given by $m_j=\Phi_je^{j+c(j-\mu)}$, where $c>0$ is such that
\begin{align}
\label{sim3}
\sum_{j\in \J}\Phi_je^{j+c(j-\mu)}=N/2.
\end{align}
Substituting the maximiser into~\eqref{ppp3} we obtain
\begin{align*}
\P\big(M=m,& \,  K=\kappa-n \, |\,  W <\e\big)\\
&=O(1)\, \sqrt{N}\prod_{j\in \J}e^{-jm_j}\exp\Big\{-\sum_{j\in \J}(c(j-\mu)-1)\Phi_je^{j+c(j-\mu)}\Big\}.
\end{align*}
Now~\eqref{sim3} implies that $c\to\infty$ and so $c(j-\mu)-1\ge c(\lambda-\mu)-1\ge 1$ for all $j$ eventually. Hence
\begin{align*}
\P\big(M=m, K=\kappa-n \, |\,  W <\e\big)
&=O(1)\,\sqrt{N}\exp\Big\{-\sum_{j\in \J}\Phi_je^{j+c(j-\mu)}\Big\}\prod_{j\in \J}e^{-jm_j}\\
&=O(1)\, \sqrt{N}e^{-\frac N 2}\prod_{j\in \J}e^{-jm_j}=o(1)\prod_{j\in \J}e^{-jm_j}.
\end{align*}
From this we can conclude that
\begin{align*}
\P\big(\langle M\rangle \ge N/2, K=\kappa-n \, |\,  W <\e\big)
&=o(1)\sum_{m\in \M}\prod_{j\in \J}e^{-jm_j}=o(1),
\end{align*}
using again that the sum is finite, similarly to~\eqref{finn}.
\end{proof}

\bigskip

{\bf Acknowledgments:} The first author is supported by EPSRC grants EP/G055068/1 and EP/IO3372X/1.

\vspace{0.3cm}

{\footnotesize
Nathana\"el Berestycki: Statistical Laboratory, DPMMS, University of Cambridge. Wilberforce Road, Cambridge CB3 0WB. United Kingdom.\\

Nina Gantert: Technische Universit\"at M\"unchen,
Fakult\"at f\"ur Mathematik,
Boltzmannstra\ss e~3,
85748~Garching bei M\"unchen.
Germany.\\

Peter M\"orters: Department of Mathematical Sciences,
University of Bath.
Claverton Down,
Bath BA2 7AY.
United Kingdom.\\

Nadia Sidorova: Department of Mathematics, University College London.
Gower Street, London WC1E 6BT. United Kingdom.}

\end{document}